\documentclass[12pt]{article}

\usepackage{geometry}

\RequirePackage{amsthm,amsmath,amsfonts,amssymb,bbm}
\RequirePackage[numbers]{natbib}
\RequirePackage[colorlinks,citecolor=blue,urlcolor=blue]{hyperref}
\RequirePackage{graphicx,float}
\usepackage{subcaption}
\RequirePackage{ulem}
\theoremstyle{plain}

\newtheorem{theorem}{Theorem}[section]
\newtheorem{rem}[theorem]{Remark}

\newtheorem{cor}[theorem]{Corollary}
\newtheorem{proposition}[theorem]{Proposition}
\theoremstyle{definition}
\newtheorem{definition}[theorem]{Definition}

\theoremstyle{remark}

\DeclareMathOperator*{\argmax}{arg\,max}
\DeclareMathOperator*{\argmin}{arg\,min}

\newcommand{\calD}{{\cal D}}

\newcommand{\calH}{{\cal H}}

\numberwithin{equation}{section}

\def\E{{\mathbb{E}}}

\def\S{{\mathbb{S}}}

\newcommand{\half}{\frac{1}{2}\:}

\newcommand{\one}{\mathbbm{1}}
\newcommand{\real}{{\mathbb{R}}}
\newcommand{\pr}{{\mathbb{P}}}

\graphicspath{{./contour_plots/}}

\title{Extreme Geometric Quantiles Under Minimal Assumptions, with a Connection to Tukey Depth}

\begin{document}

\author{Sibsankar SINGHA\textsuperscript{(a)} \thanks{This work is based on part of the doctoral thesis of the first author, completed while he was a student at TIFR–CAM under the supervision of the two other authors. It benefited greatly from mutual visits at ESSEC Business School, Paris, France, and at TIFR–CAM, Bangalore, India. The authors are grateful to the hosting institutions for their support and to the {\it Fondation des Sciences de la Mod\'elisation} (ANR-11-LABX-0023-01), and SERB--MATRICS grant MTR/2020/000629 for funding.}\\
        Marie KRATZ\textsuperscript{(b)} \\
         Sreekar VADLAMANI\textsuperscript{(c)} 
        \\[1ex]
        \small
        \textsuperscript{(a)} Télécom Paris, France; Email: sibsankar.singha@telecom-paris.fr \\
        \small
        \textsuperscript{(b)} ESSEC Business School, IDO dep.\& CREAR, France; Email: kratz@essec.edu \\
        \small
        \textsuperscript{(c)} TIFR-CAM, Bangalore, India; Email: sreekar@tifrbng.res.in
        }
        
\newgeometry{top=2cm}   % smaller top margin for first page
\maketitle

\begin{abstract}
\noindent Geometric (also known as spatial) quantiles, introduced by Chaudhury and representing one of the three principal approaches to defining multivariate quantiles, have been well studied in the literature. In this work, we focus on the extremal behaviour of these quantiles. We establish new extremal properties, namely general lower and upper bounds for the norm of extreme geometric quantiles, free of any moment conditions. We discuss the impact of such results on the characterization of distribution behaviour. Importantly, 
the lower bound can be directly linked to univariate quantiles and to halfspace (Tukey) depth central regions, highlighting a novel connection between these two fundamental notions of multivariate quantiles.
\end{abstract}

\restoregeometry
%\newpage
%\tableofcontents
\newpage

\section{Introduction}

Recent advances in high-dimensional statistics have renewed interest in tools for understanding the geometric structure of datasets. Numerous multivariate quantiles and statistical depth functions have been proposed to establish ranks and identify outliers in multivariate data.

These geometric tools, including depth functions and quantiles, provide nonparametric descriptions of a dataset in multidimensional space, making them particularly useful for statistical inference problems such as classification and regression (see e.g. \cite{Cuevas2007,Dutta2016,Hubert2010,Hallin2010,Paindaveine2012,Rousseeuw2004,Struyf1999}), outlier or anomaly detection (e.g. \cite{Staerman2022}), and applications in geometry (e.g. \cite{Kong2012}). 
A recent comparative discussion of geometric quantiles and center-outward, transport-based quantiles is given in \cite{HallinK2024}, where the distinct contour structures induced by the two approaches are analyzed.

A key question is whether the tail behaviour of a probability measure can be characterized using geometric quantiles, that is, whether they capture essential aspects of the underlying distribution. Geometric quantiles were shown by \cite{Koltchinskii1997} to uniquely identify the underlying probability measure, a property later exploited by \cite{Dhar2014} in the construction of a test statistic.
See also \cite{Konen2023} for analytical inversion methods to recover the probability measure from geometric quantiles.

It is therefore natural to look for a clearer connection between the extremal behaviour of geometric quantiles and their underlying probability measures. This motivates the present study.
Several results exist on the population-level analysis of geometric quantiles, particularly concerning their asymptotic and extreme behaviour (see \cite{Girard2017, Paindaveine2021}), which inspired our work.

In this study, we establish new extremal properties of geometric quantiles, specifically providing lower and upper bounds on the growth rate of their norms, under minimal model assumptions and without any moment conditions. 
While the upper bound is derived using classical probabilistic tools, the lower bound is obtained via a geometric approach. Importantly, the lower bound is directly related to the Tukey depth, which allows us to deduce a lower bound in terms of univariate quantiles.  This is of specific interest because our result not only provides a general rate for geometric quantiles, but also establishes connections among the two different notions of quantiles. While these notions differ considerably, it remains insightful to understand how they interact.

The general bounds are especially relevant for analyzing distributions exhibiting very heavy tails, potentially lacking first or second moments. Under standard moment assumptions, tighter results can be obtained. For example, \cite{Girard2017} showed that distributions sharing the same covariance matrix have identical first-order expansions, making them difficult to distinguish.
This motivates the use of higher-order terms to improve tail identification. Following the expansion approach proposed in \cite{Girard2017}, we extend their results and adopt their notation to ensure clarity and continuity with their work.

Finally, the lower bound results are illustrated by two numerical examples. All geometric quantile computations were implemented in Python using the algorithm developed by \cite{Chaudhury1996, Dhar2014}.

{\it Notation.} Throughout this paper, all analysis is carried out in $\mathbb{R}^d$, unless otherwise stated. The centered unit open ball and the unit sphere in $\real^d$ are denoted by $B^d$ and $S^{d-1}$, while $\langle\cdot,\cdot\rangle$ and $\|\cdot\|$ denote the Euclidean inner product and $\ell^2$-norm, respectively, in $\real^d$.

{\it Structure of the paper.} 
Section~\ref{sec:knownGeomQ} reviews definitions and properties of geometric quantiles and depth functions. 
Section~\ref{sec:gl-bound} presents our main results: general lower and upper bounds for the growth rate of the geometric quantile norm, highlighting connections with univariate quantiles and the Tukey depth. 
Section~\ref{ssec:higher-order expansion} provides further geometric and asymptotic insights under integrability conditions. 
Finally, Appendix~\ref{ssec:Tukey-depth} illustrates numerically the comparison between geometric quantiles and halfspace (Tukey) depth, 
while Appendix~\ref{app:direction} extends \cite{Girard2017} with a complementary result on the direction of geometric quantiles, including a convergence rate.

\section{Preliminaries}
\label{sec:knownGeomQ}

\begin{definition}[Geometric quantile \cite{Chaudhury1996}]
Let $X$ be a $\real^d$-valued random vector with $d>1$ such that the induced measure $\pr$ is not supported on any line,
and let $u \in \S^{d-1}$ and $0 \le \alpha < 1$.
The geometric quantile $q_X(\alpha u)$ is defined as the solution of the following optimisation problem:
    \begin{equation}
    \label{eqn:geometric-quantile-dfn}
	    q_X(\alpha u) = \argmin_{q \in \real^d}\left(\E\big[\|X-q\|-\|X\|\big]-\langle \alpha u,q \rangle\right).
    \end{equation}
\end{definition}
Existence and uniqueness of the solution of \eqref{eqn:geometric-quantile-dfn} are discussed in \cite[Section 2.1]{Chaudhury1996}. 
In the literature, $\alpha u$ is called the {\it index vector} that controls the centrality of the quantile: When $\alpha$ is close to $0$, the corresponding quantiles are close to the center of the distribution ($ q_X(0)$ being the median) and are referred to as central quantiles; when $\alpha$ is close to $1$, the corresponding quantiles are referred to as extreme quantiles. 
It was shown in \cite{Chaudhury1996} that the argument for uniqueness breaks down in $d\ge 2$ whenever the distribution of $X$ is supported on a single line. Therefore, throughout this paper, we assume that the support of each underlying probability measure is not contained in a line.

Since solving \eqref{eqn:geometric-quantile-dfn} does not admit a closed-form solution, one often looks for characterising properties of the minimiser. In particular, under the assumption that the induced measure $\pr$ is not supported on a line and is non-atomic, it was shown that for any $0 < \alpha < 1$ and $u \in \S^{d-1}$, there exists a unique $q \in \real^d$ satisfying \eqref{eqn:geometric-quantile-dfn}
if and only if the pair $(\alpha u, q) \in \real^d \times \real^d$ satisfies
\begin{equation} \label{eqn: exp of sign}
    \alpha u = -\E\left[\frac{X-q}{\|X-q\|}\right].
\end{equation}
This equivalence was established in the empirical case in \cite[Theorem 2.1.2]{Chaudhury1996},
with sufficiency shown in \cite{Girard2017}, and necessity and sufficiency obtained in
\cite[Theorem 5.1 with $\rho(t)=t$]{Konen2022}.
This characterisation of geometric quantiles is often more useful for analysis. 

Additionally, it was shown in \cite{Girard2017} that whenever geometric quantiles are well defined,
\begin{equation}\label{eqn:gq-asymp-direction}
\|q_X(\alpha u)\|\underset{\alpha \to 1} \rightarrow \infty,
\quad\text{and}\quad
\frac{q_X(\alpha v)}{\|q_X(\alpha v)\|}
\underset{(\alpha, v) \to (1,u)}\longrightarrow u,
\end{equation}
for $0<\alpha<1$ and $u,v \in \S^{d-1}$.
Another property that we will need is the affine equivariance of the geometric quantile (see \cite{Chaudhury1996}), namely, for any orthogonal matrix $\Sigma$ and any vector $\theta \in \real^d$,
\begin{equation}\label{eqn:translation-invariance}
    q_{\Sigma X + \theta}(\alpha u)
= \Sigma\, q_X(\alpha \Sigma^\top u) + \theta.
\end{equation}
Let us recall the halfspace depth and some of its properties that will be useful when defining a lower bound of extreme geometric quantile in terms of halfspace depth. 
\begin{definition}[Halfspace depth; \textnormal{\cite{Tukey1975}}] 
    \label{def:HD}
    For a random vector $X \sim \pr$ defined on $\real^d$, the halfspace depth of a point $x$ is given by: 
    $HD(x,X)= \inf \{\pr(X \in H):\,H \in \calH_x\},$ where $\calH_x$ denotes the set of all halfspaces in $\real^d$ containing $x\in\real^d$.
\end{definition}
Clearly, $HD$ takes values in $[0,1]$. Intuitively, a point with high depth is more central, while a point with low depth is relatively extremal. The point attaining the highest depth is called the halfspace median. If there is more than one point attaining the highest depth, the median is defined as the arithmetic mean of all such points.

The halfspace depth function $x \mapsto HD(x, X)$ satisfies the following properties:
\begin{itemize}
\item Affine equivariance: For any nonsingular matrix $\Sigma$ and vector $\beta \in \mathbb{R}^d$, \\
$\displaystyle HD(\Sigma x + \beta, \Sigma X + \beta) = HD(x,X)$;
\item Monotonicity along rays: The halfspace depth is nonincreasing along any ray starting from the median;
\item Vanishing at infinity: $\displaystyle \lim_{t \to \infty} HD(tx, X) = 0$;
\item Upper semicontinuity: The function $x \mapsto HD(x, X)$ is upper semicontinuous in $\mathbb{R}^d$.
\end{itemize}
The halfspace depth central region for a given $\alpha \in (0,1)$ is defined as
$$
\mathcal{D}(\alpha \mid X) = \{ x \in \mathbb{R}^d : HD(x, X) \ge \alpha \}.
$$
Whenever the context is clear, we omit the dependence on $X$, and write $\mathcal{D}(\alpha )$ instead of $\mathcal{D}(\alpha | X)$. It is shown in \cite{Tukey1975,Liu1990} that $\mathcal{D}(\alpha )$ can be represented as the intersection of all halfspaces whose probability is at least $\alpha$, i.e.,
\[
\mathcal{D}(\alpha ) = \bigcap \{ H : \pr(X \in H) \ge \alpha \}.
\]
As a consequence, $\mathcal{D}(\alpha )$ is a convex set.  

For continuous distributions, the halfspace depth function is continuous. In this case, the boundary of the central region can be written as
$$
\partial \mathcal{D}(\alpha ) = \{ x \in \mathbb{R}^d : HD(x, X) = \alpha \}.
$$
Consequently, for any $x \in \mathcal{D}(\alpha )$ and any halfspace $H$ containing $x$, we have
$$
\mathbb{P}(X \in H) \ge \alpha.
$$
Equivalently, for any unit vector $u \in \mathbb{S}^{d-1}$,
\begin{equation}\label{eqn:least probability on HD boundary}
    \pr\big( u^\top (X-x) \ge 0 \big) \ge \alpha.
\end{equation}
We will later use this property in the proof of Theorem~\ref{thm:lowerbd-q}.

%%%%%%%%%%%%%%%%%%%%%%%%%%%%%%%%
\section{General lower and upper bounds for the extremal geometric quantile norm}
\label{sec:gl-bound}
%%%%%%%%%%%%%%%%%%%%%%%%%%%%%%%%

We derive moment-free upper and lower bounds for the geometric quantile norm. The upper bound follows from a simple argument, while the more delicate lower bound has significant implications, including a connection to Tukey depth discussed in the next section.

%%%%%%%%%%%%%%
%% upper bound

\begin{theorem}[Upper bound]\label{thm:upperbd-q}
    Let $u \in S^{d-1}$ and $0<\alpha < 1$.
    \begin{enumerate}
        \item Let $k_\alpha>0$ be such that $\pr(\|X\| > k_{\alpha}) < (1-\alpha)/2$.
    Then, we have
    \begin{equation}\label{UB1}
        \|q_X(\alpha u)\| \le \frac{2\,k_\alpha \, \pr(\|X\|\le k_{\alpha})}{1-\alpha-2\pr(\|X\|\,>\, k_\alpha)}.
    \end{equation}
    \item If $\E\|X\| < \infty$, then the upper bound can be improved to:
    \begin{equation}\label{UB2}
        \|q_X(\alpha u)\| \le \frac{2 \E\|X\| }{1 - \alpha}. 
    \end{equation}
    \end{enumerate}
\end{theorem}

%\begin{remark}
This theorem extends the results of \cite{Girard2015,Girard2017} by removing the assumptions of a finite second moment and  multivariate regular variation, as is given in \eqref{UB1}. When assuming finite first moment the upper bound \eqref{UB2}, which is straightforward to compute, is asymptotically tighter than \eqref{UB1}. Moreover, it is asymptotically sharp for regularly varying tail distributions when no moment beyond the first exists, as it coincides with the rate obtained in \cite[Theorem 1(ii)]{Girard2015}. 
%%%
\\[1ex]
{\it Proof of Theorem~\ref{thm:upperbd-q}.}

Recall that $q(\alpha u) =\argmin f_{\alpha u}(q)$, where 
$$
f_{\alpha u}(q)=\E\left[ \|X-q\|-\|X\|-\langle\alpha u,q \rangle\right]. 
$$
Observe that, since $f_{\alpha u} (0) = 0$, we must have 
\begin{equation}\label{eqn:fau-upper-bd}
f_{\alpha u} (q(\alpha u)) \le f(0) = 0.
\end{equation} 

For any $q$ and $k_{\alpha}>0$, we can obtain the following lower bound by applying the triangle inequality,
\begin{eqnarray}
\label{eqn:fau-lower-bd}
f_{\alpha u}(q) &\ge & \E\left[\{\|q\|-2\|X\|\}\one_{\|X\|\leq k_{\alpha}}\right]+\E\left[ -\|q\|\one_{\|X\|>k_{\alpha}}\right]-\alpha \|q\| \nonumber \\
& \ge & \|q\|\left\{ 1 -\alpha - 2\pr(\|X\| > k_{\alpha})\right\}  - 2\,k_{\alpha} \pr(\|X\| \le k_{\alpha}).
\end{eqnarray}

Replacing $q$ with $q(\alpha u)$ in \eqref{eqn:fau-lower-bd} and combining it with \eqref{eqn:fau-upper-bd}, we conclude that
$$
\|q(\alpha u) \| \left\{ 1 -\alpha - 2\pr(\|X\| > k_{\alpha})\right\}  - 2\,k_{\alpha} \pr(\|X\| \le k_{\alpha}) \le 0,
$$
which proves the assertion of the theorem, whenever $\pr(\|X\| > k_{\alpha}) < (1-\alpha)/2$. 

Whenever $E\|X\| < \infty$, we can simply bound $f_{\alpha u}(q)$ as: 
\begin{equation}\label{eqn:fau-lower-bd2}
    f_{\alpha u}(q) \geq \|q\| - 2\E\|X\| -\alpha \|q\|.
\end{equation}
Combining \eqref{eqn:fau-upper-bd} and \eqref{eqn:fau-lower-bd2} gives \eqref{UB2}.
\hfill \qed
\\[1ex]

Theorem~\ref{thm:upperbd-q} is informative, even for distributions with no moment beyond the first one. When assuming some decay on the survival distribution of $\|X\|$, we can deduce the decay of $\|q_X(\alpha u)\|$ as $\alpha\to 1$. Namely,
%%%%%%%%
\begin{cor}\label{cor:orderUbound}~
    \begin{enumerate}
        \item When the first moment exists, the upper bound~\eqref{UB2} yields $O(1/(1-\alpha))$.
        \item 
    Suppose that $\pr(\|X\|>c) \le A\,c^{-\beta}$ for some $\beta>0$ and all sufficiently large $c$.  

    If $(1-\alpha)$ is chosen small enough so that $(1-\alpha)^{-1} > A$, then
$$
    \|q_X(\alpha u)\| = O\left((1-\alpha)^{-\left(1+\frac{1}{\beta}\right)\frac{\beta}{\beta-\epsilon}}\right), \,\,\,\forall \,\, 0< \epsilon < \beta .
$$
 \end{enumerate}
\end{cor}

This is a sharp bound whenever moments of order larger than $1$ do not exist. When the first moment exists, the bound \eqref{UB2} becomes sharper, yielding a $O(1/(1-\alpha))$ rate. Recall that, when assuming moments of order at least $2$, the exact rate is $O(1/\sqrt{(1-\alpha)})$ as shown in \cite{Girard2017}. We will return to this point in Remark~\ref{rk:MRV}, where both the upper and lower bounds are discussed in the multivariate regularly varying (MRV) case. However, it is worth already noticing that this lower bound holds in heavy-tailed settings beyond the MRV framework (e.g., allowing marginals with different tail heaviness), as is illustrated in Figure \ref{fig:correlated} in Section~\ref{ssec:Tukey-depth}. 

As established in Theorem~\ref{thm:upperbd-q}, and made explicit in the heavy-tailed setting (Corollary~\ref{cor:orderUbound}), the upper bound is generally not sharp, except in specific cases such as those discussed in Proposition~\ref{prop:bounds MRV}.
\\[1ex]
%%%%%%%
{\it Proof of Corollary~\ref{cor:orderUbound}.}
%%%
Let $\pr(\|X\|>c) \le A\,c^{-\beta}$, for some $\beta>0$ and large enough $c$. Set $(1-\alpha)/2 = A c^{-\beta+\epsilon}$. Then, $k_{\alpha} = (2A)^{1/(\beta-\epsilon)}(1-\alpha)^{-1/(\beta-\epsilon)}$. With this notation, we have
\begin{eqnarray*}
\|q(\alpha u)\| & \le & \frac{(2A)^{1/(\beta-\epsilon)}(1-\alpha)^{-1/(\beta-\epsilon)} \pr(\|X\|\le k_{\alpha})}{\frac{(1-\alpha)}{2} - \pr(\|X\|>k_{\alpha})}\\
&\le &  \frac{(2A)^{1/(\beta-\epsilon)}(1-\alpha)^{-1/(\beta-\epsilon)}}{\frac{(1-\alpha)}{2} - \frac{A^{\epsilon/(\beta-\epsilon)}(1-\alpha)^{\beta/(\beta-\epsilon)}}{2^{\beta/(\beta-\epsilon)}}}
 =  \frac{(2A)^{1/(\beta-\epsilon)}(1-\alpha)^{-1/(\beta-\epsilon)}}{(1-\alpha)^{\beta/(\beta-\epsilon)}\left[ \frac{(1-\alpha)^{-\epsilon/(\beta-\epsilon)}}2 - \frac{A^{\epsilon/(\beta-\epsilon)}}{2^{\beta/(\beta-\epsilon)}}\right]}.
\end{eqnarray*}
Choosing $(1-\alpha)$ small enough so that $(1-\alpha)^{-1} >A$, we have
$$
\|q_X(\alpha u)\| \le c (1-\alpha)^{-\frac{\beta+1}{\beta-\epsilon}} = c (1-\alpha)^{-\left( 1+ \frac1{\beta}\right)\frac{\beta}{(\beta-\epsilon)}},\,\,\forall \epsilon>0. \qquad\qquad\qquad\qquad\quad \Box
$$

%%%%%%%%%%%%%%
%% lower bound

We now address the lower bound. Since the quantile norm $\|q_X(\alpha u)\|$ diverges to infinity, it is also relevant to derive such a bound. We obtain it via a geometry-based approach. First, we show that the extreme geometric quantile region contains the halfspace depth region at a given level. We then deduce a lower bound in terms of univariate quantiles.
\begin{theorem}[]\label{thm:lowerbd-q}
Let $X$ be a continuous, fully supported random vector (as a consequence, its distribution is not concentrated on any line). Let $q_X(\alpha u)$ be the geometric quantile for $0 < \alpha < 1$ and $u\in \S^{d-1}$. Then, we have
\begin{equation}\label{eq:lowerBd}
    \left\{x \in \real^d:  HD\left(x| X\right) \ge \frac{1 - \alpha^2}{M_\gamma}\right\} \subseteq \left\{ q_X(\beta u) : 0 \le \beta \le \alpha, u \in \S^{d-1}\right\},
\end{equation}
where the parameter $\alpha$ satisfies $\alpha^2 \in \left( 1 - \tfrac{M}{d +1},\, 1 \right)$ and $0 < M_\gamma < 1$ is defined, for any $0<\gamma<1$, by
\begin{equation}\label{eqn:constant_c}
    M_\gamma = (1 - \gamma) \,\inf_{u\in \S^{d-1}}  \pr(u^\top X \ge \| X \| \sqrt{1 - \gamma^2}).
\end{equation}
\end{theorem}

\begin{rem}[On the geometric constant $M_\gamma$]
The constant $M_{\gamma}$ defined in \eqref{eqn:constant_c}
is the key geometric quantity linking geometric quantiles and Tukey depth. It acts as a scaling penalty on the depth level in the inclusion given by \eqref{eq:lowerBd}.
A smaller $M_\gamma$ increases the depth threshold $(1-\alpha^2)/M_\gamma$, thereby shrinking the guaranteed Tukey region and weakening the bound.

%%%%%%%%%%%%%%%%%%
%\textbf{Geometric interpretation.}
Geometrically, for fixed $\gamma\in(0,1)$, the event $\{u^\top X \ge \|X\|\sqrt{1-\gamma^2}\}$ corresponds to $X$ lying inside a circular cone centered at direction $u$ with half-angle $ \theta = \arcsin(\gamma)$. Hence $M_\gamma$ measures the minimal angular probability mass assigned to such a cone, uniformly over all orientations $u$. It quantifies how much probability must lie in every polar cap of aperture $\theta$.
If the distribution spreads mass evenly in all directions, $M_\gamma$ is relatively large. If there exists a direction with very little probability mass (strong anisotropy or skewness), the infimum becomes small, reducing $M_\gamma$.
%%%%%%%%%%%
%\textbf{Rotationally invariant distributions.}
For rotationally invariant distributions (e.g.,\ multivariate normal or multivariate Cauchy), the angular distribution is uniform on the sphere. Then, $M_\gamma$ admits the explicit form
\[
M_\gamma
=
\frac{1-\gamma}{2}
\, I_{\gamma^2}\!\left(\frac{d-1}{2}, \frac{1}{2}\right),
\]
where $I_x(a,b)$ denotes the regularized incomplete beta function.

Importantly, this expression depends only on the dimension $d$ and not on the radial distribution. Consequently, Gaussian and Cauchy distributions share the same value of $M_\gamma$. The scaling factor $M_{\gamma}$ therefore depends purely on directional geometry rather than on tail behavior.

%%%%%%%%%%
%\textbf{Anisotropy and high dimension.}
For asymmetric distributions, the infimum in \eqref{eqn:constant_c} is attained in the direction where the angular mass is the smallest. In such cases $M_\gamma$ decreases, making the bound more conservative. Moreover, the beta-function representation shows that, as $d\to\infty$,
\[
I_{\gamma^2}\!\left(\frac{d-1}{2},\frac12\right)\to 0,
\]
implying that $M_\gamma \to 0 $ as $d\to\infty$, illustrating a form of curse of dimensionality.
\end{rem}

\begin{rem}[On the admissible range of $\alpha$]
The parameter $\alpha$ must be chosen so that
\[
\frac{1-\alpha^2}{M_\gamma} < \frac{1}{d+1}.
\]
This condition, by way of the center point theorem (cf. \cite{Rado1946}), guarantees that the halfspace depth region
\[
\left\{x\in\real^d : HD(x\mid X)\ge \frac{1-\alpha^2}{M_\gamma}\right\}
\]
is nonempty, and hence provides a nontrivial inclusion in \eqref{eq:lowerBd}.
This requirement is conservative as it relies only on the universal lower bound $1/(d+1)$ for the maximal depth. For specific distributional classes, the maximum depth may be strictly larger, thereby enlarging the admissible range of $\alpha$.

For instance, if the distribution of $X$ is halfspace-symmetric, then its maximal depth equals $1/2$. In this case, nonemptiness only requires
\[
\frac{1-\alpha^2}{M_\gamma} < \frac{1}{2}
\qquad\text{or, equivalently,}\qquad
1 - \frac{M_\gamma}{2} < \alpha^2 < 1.
\]
In summary, the admissible range of $\alpha$ is determined jointly by the geometric constant $M_\gamma$ and the maximal achievable depth of the distribution. 
\end{rem}

{\it Proof of Theorem~\ref{thm:lowerbd-q}.}
We proceed via a geometry-based approach.
The central argument for the proof is given by the following proposition.
\begin{proposition}
    \label{propn:upper-bound-of-U_q}
Let $q \in \mathbb{R}^d$ and define
$$
U_q = \frac{X-q}{\|X-q\|},
$$
where $X$ is a random vector in $\mathbb{R}^d$ with no atoms. Let $\bar U_q$ be an independent copy of $U_q$. Then,
$$
\|\mathbb{E}(U_q)\|^2 = 1 - \frac{1}{2}\,\mathbb{E}\|U_q - \bar U_q\|^2.
$$
\end{proposition}

This proposition provides a geometric control on the index vector $\mathbb{E}(U_q)$, by obtaining analytical bounds that hold uniformly over $q$.

\noindent{\it Proof of Proposition~\ref{propn:upper-bound-of-U_q}.} Since $\|U_q\| = 1$ almost surely, for an independent copy $\bar U_q$, we can write
$$
\mathbb{E}\|U_q - \bar U_q\|^2 
% = \mathbb{E}\Big(\|U_q\|^2 + \|\bar U_q\|^2 - 2\, U_q \cdot \bar U_q\Big) 
= 2 - 2\,\mathbb{E}\left(\langle U_q,\bar U_q\rangle\right) = 2 - 2\,\langle\mathbb{E}(U_q),\mathbb{E}(\bar U_q)\rangle
= 2 - 2\,\|\mathbb{E}(U_q)\|^2,
$$
hence the proposition. \hfill $\Box$

\medskip
Proposition \ref{propn:upper-bound-of-U_q} implies that any lower bound on $\mathbb{E}\|U_q - \bar U_q\|^2$ yields an upper bound on $\|\mathbb{E}(U_q)\|^2$. 

If $A,B \subset S^{d-1}$ are two measurable subsets with $\mathbb{P}(U_q \in A)>0$, $\mathbb{P}(U_q \in B) > 0$, and $\displaystyle\delta = \min_{u \in A, v \in B} \|u-v\| > 0$, then
$$
\mathbb{E}\|U_q - \bar U_q\|^2 \ge \mathbb{E}\big[\|U_q - \bar U_q\|^2 \mathbf{1}_{\{U_q \in A, \bar U_q \in B\}}\big] \ge \delta^2 \,\mathbb{P}(U_q \in A)\,\mathbb{P}(U_q \in B),
$$
and, consequently,
\begin{equation}\label{ineq: general-inequality}
    \|\mathbb{E}(U_q)\|^2 \le 1 - \frac{1}{2}\, \delta^2 \,\mathbb{P}(U_q \in A)\,\mathbb{P}(U_q \in B) < 1.
\end{equation}
We will explicitly construct the sets $A$ and $B$.\\ 
Let us denote by $d_g(u,v) = \arccos \langle u,v \rangle$, the {\it great-circle distance} between two points $u, v \in \S^{d-1}$. The function $d_g$ defines a metric on the space $\S^{d-1}$, satisfying $\displaystyle 0 \le d_g(u,v) \le \pi$ and which is related to the Euclidean distance by 
$\displaystyle \|u-v\| = 2 \sin(d_g(u,v)/2)$.
Define, for any $q \neq 0$, 
$$
A_q^\gamma = \left\{ w \in \S^{d-1} : d_g\Big(-\frac{q}{\|q\|}, w\Big) \le \arccos(\sqrt{1-\gamma^2}) \right\}\quad\text{with}\; 0<\gamma<1,
$$ 
% so that $\arccos(\sqrt{1-\gamma^2}) \in (0;\pi/2)$
and $$
B_q = \left\{ w \in \S^{d-1} : d_g\Big(\frac{q}{\|q\|}, w\Big) \le \frac{\pi}{2} \right\}.
$$

Clearly, the sets $A_q^\gamma$ and $B_q$ are disjoint. Indeed, suppose they were not and 
let $u \in A_q^\gamma \cap B_q$. Then we would obtain by triangle inequality
$$
\pi = d_g\!\left(-\frac{q}{|q|}, \frac{q}{|q|}\right)
\le d_g\!\left(-\frac{q}{|q|}, u\right) + d_g\!\left(u, \frac{q}{|q|}\right)
< \pi,
$$
which is a contradiction. 

Now,  let us estimate $\delta$ in \eqref{ineq: general-inequality}, with $A=A_q^\gamma$ and $B=B_q$. For $u \in A_q^\gamma$ and $v \in B_q$, applying  
the triangle inequality for $d_g$, we obtain
\begin{equation*}
\begin{aligned}
\pi = d_g\Big(-\frac{q}{\|q\|}, \frac{q}{\|q\|}\Big)
&\le d_g\Big(-\frac{q}{\|q\|}, u\Big) + d_g(u,v) + d_g\Big(v, \frac{q}{\|q\|}\Big) \\
&\le \arccos(\sqrt{1-\gamma^2}) + d_g(u,v) + \frac{\pi}{2},
\end{aligned}
\end{equation*}
which implies \,
$\displaystyle d_g(u,v) \ge \frac{\pi}{2} - \arccos(\sqrt{1-\gamma^2}) = \arcsin(\sqrt{1-\gamma^2})$.
Hence,
$$
\delta = \min_{u \in A_q^\gamma, v \in B_q} \|u-v\| = \sqrt{2\Big(1-\cos(d_g( u,v))\Big)} \ge 
%2 \sin\!\left(\frac{\arcsin(\sqrt{1-\gamma^2})}{2}\right) = 
\sqrt{2(1-\gamma)}.
$$
Substituting this lower bound of $\delta$ into the inequality~\eqref{ineq: general-inequality}, we obtain
\begin{equation}\label{eqn:estimate_index_vector}
\!\!\!\!\!\!\! \|\mathbb{E}(U_q)\|^2 
\le 1 - \frac{1}{2} \,\delta^2 \, \mathbb{P}(U_q \in A_q^\gamma)\,\mathbb{P}(U_q \in B_q) 
= 1 - \left(1 - \gamma\right) \mathbb{P}(U_q \in A_q^\gamma)\,\mathbb{P}(U_q \in B_q). \,
%&< 1 - \left(1 - \gamma\right) \mathbb{P}(\|X\| \le \gamma \|q\|)\,\mathbb{P}(X^\top q > \|q\|^2),
\end{equation}
Next, we show that there exists a constant $M_\gamma > 0$ such that 
$$
0 < M_\gamma \le (1 - \gamma) \pr(U_q \in A_q^\gamma), \quad \text{for all } q \neq 0.
$$
First, observe that for $\frac{x}{\|x\|} \in A_q^\gamma $, 
$$
d_g\Big(\frac{x}{\|x\|}, -\frac{q}{\|q\|}\Big) \le \arccos\!\left(\sqrt{1-\gamma^2}\right)
\,\Rightarrow \,
d_g\Big(\frac{x-q}{\|x-q\|}, -\frac{q}{\|q\|}\Big) \le \arccos\!\left(\sqrt{1-\gamma^2}\right),
$$
and therefore
$$
\pr(U_0 \in A_q^\gamma) \le \pr(U_q \in A_q^\gamma).
$$
Let us define 
\begin{equation}\label{dfn:M}
    M_\gamma = (1 - \gamma) \inf_{q \in \S^{d-1}} \pr(U_0 \in A_q^\gamma),
\end{equation}
then $M_\gamma >0$, because $X$ is fully supported on $\mathbb{R}^d$. Finally, observe that by the definition of $U_q$, 
$$
U_q \in B_q \quad \text{if and only if}\quad \left\langle \frac{X - q}{\|X - q\|}, \frac{q}{\|q\|} \right\rangle \ge 0, 
$$
which is equivalent to $\displaystyle \left\langle X - q, \frac{q}{\|q\|}\right\rangle \ge 0$. Therefore, using \eqref{eqn:estimate_index_vector} and \eqref{dfn:M}, we can write
$$
 \|\mathbb{E}(U_q)\|^2 < 1 - M_\gamma\,\cdot\,\pr\left(\left\langle X - q, \frac{q}{\|q\|}\right\rangle \ge 0 \right).
$$
Now, let $ q \in \mathcal{D}(\frac{1-\alpha^2}{M_\gamma})$ where $\alpha^2 \in (1-\frac{M_\gamma}{d+1}, 1)$. Then, from \eqref{eqn:least probability on HD boundary}, it follows that
$\pr\left(\left\langle X - q, \frac{q}{\|q\|}\right\rangle \ge 0 \right) \ge \frac{1-\alpha^2}{M_\gamma}$, which implies
$$
\sup_{q\in D(\frac{1-\alpha^2}{M_\gamma}| X)} \|\mathbb{E}(U_q)\|^2 
< 1 - M_\gamma \frac{1-\alpha^2}{M_\gamma}
= \alpha^2.
$$
In other words, 
\begin{equation}\label{eq:D-Eu}
\calD\left(\frac{1-\alpha^2}{M_\gamma}| X\right) \subseteq \left\{q\in \real^d : \| \E (U_q)\| \le \alpha \right\}.
\end{equation}   
{
Moreover, using the characterization \eqref{eqn: exp of sign} of the geometric quantile, for any $0 < \alpha < 1$ we have
\[
\left\{ q \in \mathbb{R}^d : \| \mathbb{E}(U_q) \| \le \alpha \right\} 
= \left\{ q_X(\beta u) : 0 \le \beta \le \alpha, \, u \in \mathbb{S}^{d-1} \right\},
\]
where $q_X(\beta u)$ denotes the geometric quantile defined in \eqref{eqn:geometric-quantile-dfn}. 
Combining this with \eqref{eq:D-Eu} completes the proof of Theorem~\ref{thm:lowerbd-q}.
}
\hfill $\Box$\\[1ex]
%\end{proof}

The set inclusion in Theorem~\ref{thm:lowerbd-q} naturally induces a lower bound on the norm of the geometric quantile, expressed in terms of a univariate quantile, as stated in the following theorem. 
%%%%%
\begin{theorem}\label{thm:lowerbd-q-univ}
Let $X$ be a continuous random vector fully supported on $\real^d$. Let  $\theta \in \argmax HD(x|X)$ be a point of highest depth ({\it i.e.}, halfspace median). Then
\begin{equation}\label{eq:lowerBd-univariate}
\| q_X(\alpha u) - \theta \|
\;\ge\;
\min_{\|u\| = 1}
\left|
Q_{u^\top X}\!\left( 1 - \frac{1 - \alpha^2}{M_\gamma} \right)
- u^\top \theta
\right|,
\end{equation}
where $Q_{u^\top X}(\beta)
= \inf \{ t \in \mathbb{R} : \mathbb{P}(u^\top X > t) < 1 - \beta \}$ (with $0\le \beta <1$)
denotes the univariate quantile function of the projection $u^\top X$, and where the parameter $\alpha$ satisfies $\alpha^2 \in \left( 1 - \tfrac{M_\gamma}{d +1},\, 1 \right)$, with $0 < M_\gamma < 1$ defined in \eqref{eqn:constant_c}.
\end{theorem}

The lower bound in \eqref{eq:lowerBd-univariate} reduces the multivariate problem to a collection of one-dimensional constraints involving projected distributions. For any direction $u \in \mathbb{S}^{d-1}$, the displacement of the geometric quantile from the halfspace median $\theta$ is bounded below by the univariate quantile of $u^\top X$; taking the minimum over $u$ yields the stated bound. The bound is therefore explicit and computationally tractable. It is attained in the direction for which the projected distribution reaches the given depth level with the smallest displacement, that is, in the marginal exhibiting the lightest tail at the relevant level.
\\
Consequently, by \eqref{eq:lowerBd}, the geometric quantile cannot grow more slowly in any direction than dictated by the associated univariate depth constraint. When second order moments exist, then an application of Markov's inequality yields 
$$
Q_{u^\top X}\left( 1 - \frac{1 - \alpha^2}{M_\gamma} \right) = O\!\left((1-\alpha)^{-1/2}\right), \quad\text{uniformly in $u$},
$$
whereas the exact asymptotic rate established in \cite{Girard2017} is
\[
(1-\alpha)^{-1/2}\,\frac{1}{2}\bigl(\operatorname{tr}\Sigma - \langle \Sigma u,u\rangle\bigr).
\]
Hence, our lower bound is, in general, conservative relative to the exact rate. Its principal advantage lies in its uniformity over all directions and in the fact that it does not rely on moment assumptions or structural conditions such as multivariate regular variation (MRV) or existence of the density. 

Moreover, under MRV assumptions, the bound recovers the sharp rate obtained in \cite{Girard2015}, as detailed in the subsequent remark and proposition.

\begin{rem}{\bf Sharpness of the bounds under MRV.}
\label{rk:MRV}

%%%%%%%%%%%%%%%
As already noted, the upper bound given in Corollary~\ref{cor:orderUbound} is of order
$(1-\alpha)^{-\left(1+\frac{1}{\beta}\right)\frac{\beta}{\beta-\epsilon}}$,
which can be much larger than the MRV rate $(1-\alpha)^{-1/\beta}$ derived in \cite{Girard2015}. 
Since $\frac{\beta+1}{\beta-\epsilon} > \frac{1}{\beta}$, our upper bound diverges faster as $\alpha \uparrow 1$, highlighting the conservativeness of this general estimate.  

By contrast, the general lower bound established in Theorem~\ref{thm:lowerbd-q} is remarkably sharp: it matches the exact polynomial growth rate $(1-\lambda)^{-1/\beta}$ obtained in \cite{Girard2015} for distributions admitting a density but lacking moments of order higher than two. In our setting, the effective depth parameter corresponds to $\lambda=\alpha^2$, so that the governing quantity becomes $(1-\alpha^2)$. This identification is also consistent with the refined results in \cite{Girard2017}.
\end{rem}

We summarise these results under MRV in the following proposition.
\begin{proposition} \label{prop:bounds MRV}
Under the MRV assumptions, the following statements hold:
\begin{enumerate}
    \item Under MRV(1), the lower and upper bounds coincide and yield the rate 
    \[
        O\!\left( (1-\alpha)^{-1} \right),
    \]
    which matches the exact rate established in \cite{Girard2015}.
    
    \item Under MRV($\beta$) with $\beta>0$, the lower bound yields the rate 
    \[
        O\!\left( (1-\alpha)^{-\beta} \right),
    \]
    which coincides with the exact rate in \cite{Girard2015} for $\beta \leq 2$. 
    For $\beta \neq 1$, however, the upper bound becomes overly conservative, as discussed in Corollary~3.2.
\end{enumerate}
\end{proposition}

%\begin{proof}
{\it Proof of Theorem~\ref{thm:lowerbd-q-univ}}

    By Theorem~\ref{thm:lowerbd-q}, we have
    \begin{equation*}
        \theta \in \left\{x \in \real^d:  HD\left(x| X\right) > \frac{1 - \alpha^2}{M_\gamma}\right\} \;\Rightarrow \; \theta \in \left\{ q_X(\beta u) : 0 \le \beta \le \alpha, u \in \S^{d-1}\right\}.
    \end{equation*}
Therefore, any line segment joining $\theta$ and $q_X(\alpha u)$ must intersect the closed curve \\ $\left\{x \in \real^d:  HD\left(x| X\right) = \frac{1 - \alpha^2}{M_\gamma}\right\}$. As a result,
\begin{equation*}
    \| q_X(\alpha u) - \theta \|
\;\ge\; \min_{x\in \partial \calD (\frac{1-\alpha^2}{M_\gamma}|X)} \| x - \theta\|.
\end{equation*}
Thus, it suffices to show that 
\begin{equation}\label{eqn: depth_equals_dirQuantile}
    \min_{x\in \partial \calD (\frac{1-\alpha^2}{M_\gamma}| X)} \| x - \theta\| = \min_{\|u\| = 1}
\left|
Q_{u^\top X}\!\left( 1 - \frac{1 - \alpha^2}{M_\gamma} \right)
- u^\top \theta
\right|.
\end{equation}
Notice that it is enough to establish \eqref{eqn: depth_equals_dirQuantile} for $\theta = 0$, due to the translation invariance of geometric quantile \eqref{eqn:translation-invariance} and directional quantile.
Indeed, for $\theta \neq 0$, set $Y = X - \theta$. 
In this case, the result follows from the implication
\[
x \in \partial \mathcal{D}\!\left(\frac{1-\alpha^2}{M_\gamma} \mid Y\right)
\;\Rightarrow\;
x + \theta \in \partial \mathcal{D}\!\left(\frac{1-\alpha^2}{M_\gamma} \mid X\right),
\]
together with the translation invariance of directional quantiles:
\[
Q_{u^\top Y}(\cdot) = Q_{u^\top X}(\cdot) - u^\top \theta .
\]

By hypothesis, the point $\theta = 0$ attains the maximum depth, which is at least $\frac{1}{d+1}$ (see \cite{Donoho1992}). 
Consequently, any halfspace containing $0$ has probability at least $\frac{1}{d+1}$. 
Since by assumption $\frac{1-\alpha^2}{M_\gamma} < \frac{1}{d+1}$, it follows that, for any $u \in \mathbb{S}^{d-1}$,
\[
\pr(u^\top X \ge 0) \ge \frac{1}{d+1} > \frac{1-\alpha^2}{M_\gamma}.
\]
Hence,  
\[
0 \le \, Q_{u^\top X}\!\left(1 - \frac{1-\alpha^2}{M_\gamma}\right) \quad \text{for all } u \in \mathbb{S}^{d-1}.
\]
Moreover, $Q_{u^T X}(\cdot)$ is continuous in $u$ (see \cite{Kong2012}), , so the minimum on the right-hand side of \eqref{eqn: depth_equals_dirQuantile} is attained at some $u_* \in \S^{d-1}$. Consequently, for any $v \neq u_*$, we have 
$$
0 \le Q_{u_*^\top X}\!\left( 1 - \frac{1 - \alpha^2}{M_\gamma} \right) \le Q_{v^\top X}\!\left( 1 - \frac{1 - \alpha^2}{M_\gamma} \right).
$$
Since $X$ is continuous, by the definition of the univariate quantile we have
\begin{align}
\frac{1-\alpha^2}{M_\gamma}
&=
\pr\!\left(u_*^\top X \ge Q_{u_*^\top X}\!\left(1-\frac{1-\alpha^2}{M_\gamma}\right)\right) 
=
\pr\!\left(v^\top X \ge Q_{v^\top X}\!\left(1-\frac{1-\alpha^2}{M_\gamma}\right)\right) \label{eqn1:min_dir_quantile}\\
&\le
\pr\!\left(v^\top X \ge Q_{u_*^\top X}\!\left(1-\frac{1-\alpha^2}{M_\gamma}\right)\right) \quad \text{(since $Q_{u_*^\top X} \le Q_{v^\top X}$ by definition of $u_*$)}  \notag\\
&\le \pr\!\left(v^\top X \ge (v^\top u_*)\, Q_{u_*^\top X}\!\left(1-\frac{1-\alpha^2}{M_\gamma}\right)\right), \label{eqn:min_dir_quantile}
\end{align}
where the inequalities hold for all $v \in \mathbb{S}^{d-1}$ . 
To justify the last inequality, decompose $X$ into its components parallel and orthogonal to $u_*$, 
$\displaystyle X = (u_*^\top X) u_* + X_\perp$, with $\displaystyle X_\perp := X - (u_*^\top X) u_*$,
so that $u_*^\top X_\perp = 0$. Then $\displaystyle v^\top X = (v^\top u_*)(u_*^\top X) + v^\top X_\perp$, from which \eqref{eqn:min_dir_quantile} follows.

Define
\[
y = Q_{u_*^\top X}\!\left(1-\frac{1-\alpha^2}{M_\gamma}\right) u_* .
\]
Then \eqref{eqn:min_dir_quantile}, combined with \eqref{eqn1:min_dir_quantile}, can be rewritten as
\[
\pr\big(u_*^\top (X - y) \ge 0\big) = \frac{1-\alpha^2}{M_\gamma} 
\le \pr\big(v^\top (X - y) \ge 0\big), 
\quad \text{for all } v \in \mathbb{S}^{d-1}.
\]
Hence, the optimal direction in the halfspace depth of $y$ is attained at $u_*$. In particular,
\[
HD(y \mid X)
=
\pr\!\left(u_*^\top (X - y) \ge 0\right)
=
\frac{1-\alpha^2}{M_\gamma}.
\]
This implies that $y \in \partial \mathcal{D}\!\left(\frac{1-\alpha^2}{M_\gamma} \mid X\right)$. Consequently,
\begin{equation}\label{eqn:depth_radius_to_dirQuantile}
\min_{x \in \partial \mathcal{D}\!\left(\frac{1-\alpha^2}{M_\gamma} \mid X\right)} \|x\|
\,\le \,
\|y\|
=
\min_{\|u\| = 1}
\left|
Q_{u^\top X}\!\left(1-\frac{1-\alpha^2}{M_\gamma}\right)
\right|.
\end{equation}
Next, we show that the inequality \eqref{eqn:depth_radius_to_dirQuantile} is in fact an equality. 
Assume, for contradiction, that the inequality is strict. 
Then there exists $z \in \partial \mathcal{D}\!\left(\frac{1-\alpha^2}{M_\gamma}\right)$ such that $\|z\| < \|y\|$, 
where $y$ corresponds to the directional quantile in direction $u_*$. 
Let $w \in \mathbb{S}^{d-1}$ be such that
\[
HD(z \mid X) = \Pr\!\left(w^\top (X - z) \ge 0\right) = \frac{1-\alpha^2}{M_\gamma}.
\]
Then
\[
Q_{w^\top X}\!\left(1 - \frac{1-\alpha^2}{M_\gamma}\right) = w^\top z \le \|z\| < \|y\| = Q_{u_*^\top X}\!\left(1 - \frac{1-\alpha^2}{M_\gamma}\right),
\]
which contradicts the fact that $u_*$ minimizes $\big| Q_{u^\top X}\!(1 - \frac{1-\alpha^2}{M_\gamma}) \big|$ over $\|u\| = 1$.

Therefore, any $z \in \partial \mathcal{D}\!\left(\frac{1-\alpha^2}{M_\gamma}\right)$ satisfies $\|z\| \ge \|y\|$.
In particular,
\begin{equation}
\label{eqn:opt_depth_radius}
\min_{x \in \partial \mathcal{D}\!\left(\frac{1-\alpha^2}{M_\gamma}\right)} \|x\| \ge \|y\|.
\end{equation}
Hence, the theorem follows by combining \eqref{eqn:depth_radius_to_dirQuantile} and \eqref{eqn:opt_depth_radius}. \hfill $\Box$\\[1ex]
% \end{proof}
%%
More generally, our geometry-based approach extends these results by providing lower bounds that do not rely on moments or smoothness assumptions, and therefore apply to a broader class of distributions. 
While the preceding results capture the effect of directional heterogeneity up to second order, they do not reflect the influence of skewness or higher-order tail behaviour, which we now investigate under appropriate moment conditions in the next section.
 
\section{Under integrability conditions}
\label{ssec:higher-order expansion} 

The first-order asymptotic behaviour of the geometric quantile was established in \cite{Girard2017}. When $\E\|X\|^2<\infty$, it holds that
$$
\lim_{\alpha\to 1}\|q_X(\alpha u)\|^2(1-\alpha) =
\frac{1}{2}\bigl(\mathrm{tr}\,\Sigma - \langle \Sigma u, u \rangle\bigr),
$$
where $\Sigma$ is the covariance matrix of $X$. 
This leading term reflects anisotropy through the covariance structure and does not capture higher-order features such as skewness or tail asymmetry. 
In particular, distributions sharing the same covariance matrix have identical first-order behaviour.

To refine this analysis, we extend the asymptotic expansion approach of \cite{Girard2017} beyond the first order. Under stronger moment assumptions, higher-order terms reveal the influence of skewness and tail behaviour. 
In particular, assuming $\E\|X\|^3<\infty$, we obtain in Theorem~\ref{thm:higher-order} the following third-order expansion. The methodology naturally extends to obtain even higher-order expansions under stronger moment assumptions.

\begin{theorem}\label{thm:higher-order}
Let $u \in S^{d-1}$ and $\pr$ be a probability measure on $\real^d$ with no atom such that its support is not concentrated on any line. 
    If $\E\|X\|^3 < \infty$, then we have 
    \begin{eqnarray}\label{eqn:3rd-order}		
        &&\!\!\!\!\!\!\!\!\!\!\!\!\!\!\!\!\!\!\!\!\!  \|q_X(\alpha u)\|\Big[\|q_X(\alpha u)\|^2 (1 - \alpha ) -  \half (\text{tr } \Sigma - \langle\Sigma u,u\rangle)\Big] \; \underset{\alpha\to 1}{\longrightarrow} \nonumber \\
        && \E\Big(\langle X,u\rangle \|X-\langle X,u\rangle\|^2\Big) - \Big\langle (2\Sigma+uu^T) u,(I - uu^T)\E X\Big\rangle .
	\end{eqnarray}
%\end{enumerate}
\end{theorem}

\begin{rem}
If $X$ has mean zero, then Equation~\eqref{eqn:3rd-order} simplifies to:
$$\|q_{X}(\alpha u)\|\left(\|q_{X}(\alpha u)\|^2 (1 - \alpha ) -  \half (\text{tr } \Sigma - \langle\Sigma u,u\rangle)\right) \; \underset{\alpha\to 1}{\longrightarrow} \E\left( \langle X,u\rangle \|X-\langle X,u\rangle\|^2\right).
$$
Additionally, if we suppose that the projection of $X$ on $u$ is independent of its projection onto the orthogonal complement of $u$, then we have:
$$\|q_{X}(\alpha u)\|\left(\|q_{X}(\alpha u)\|^2 (1 - \alpha ) -  \half (\text{tr } \Sigma - \langle\Sigma u,u\rangle)\right) \; \underset{\alpha\to 1}{\longrightarrow} 0.
$$
\end{rem}

Theorem~\ref{thm:higher-order} complements Lemmas 6.2–6.5 of \cite{Girard2017}, by differentiating between distributions with very heavy tails (for which no second moment exists) and those with moderately heavy tails (for which the third moment exists).
\\[1ex]
Note that the appendix presents a complementary result on the direction of geometric quantiles, included for completeness and to provide a fuller account of the findings in \cite{Girard2017}, which we extend by establishing a rate of convergence.
\\[1ex]
{\it Proof of Theorem~\ref{thm:higher-order}}

Let us denote by $\{u,w_1,...,w_{d-1}\}$ the orthonormal basis of $\real^d$. Let $b(\alpha)$ and $\{\beta_k(\alpha)\}_{k=1}^{d-1}$ be real numbers defined by 
\begin{equation}
\label{eqn:decomposition}
    \frac{q(\alpha u)}{\|q(\alpha u)\|}=b(\alpha )u+\sum _{k=1}^{d-1}\beta_k(\alpha)w_k.
\end{equation}
Using Lemmas 6.2 \& 6.3 in \cite{Girard2017}, we can write
\begin{eqnarray}\label{eqn:10}
& &\|q(\alpha u)\| \Bigg[\|q(\alpha u)\|^2\beta_k^2 (\alpha ) - |\E\langle X,w_k\rangle |^2\Bigg ]\nonumber\\
&=& \Bigg[\|q(\alpha u)\| \Big(\|q(\alpha u)\|\beta_k (\alpha ) - |\E\langle X,w_k\rangle |\Big)\Bigg ]\Bigg[\|q(\alpha u)\|\beta_k (\alpha ) + \E\langle X,w_k\rangle \Bigg] \nonumber\\
&\underset{\alpha\to 1}{\longrightarrow} & \text{cov} (\langle X,u\rangle , \langle X,w_k\rangle )\,\, 2\E\langle X,w_k\rangle.
\end{eqnarray}
Using the decomposition \eqref{eqn:decomposition} and \cite[Proposition 6.3]{Girard2017}
under the assumption $\E\|X\|^3<\infty$, we have
\begin{eqnarray}\label{eqn:11}
&& \|q(\alpha u)\| \Bigg (\|q(\alpha u)^2 (1-\alpha b(\alpha))-\half \E\|X-\langle X,u\rangle u\|^2\Bigg )\\
& \underset{\alpha\to 1}{\longrightarrow}& \E\Bigg (\langle X,u\rangle \Big [ \|X-\langle X,u\rangle u\|^2-\langle X,\E(X-\langle X,u\rangle u)\rangle  \Big]\Bigg ) \,
 \stackrel{\Delta}{=}\, f_X(u). \nonumber
\end{eqnarray}
Therefore, from Equations \eqref{eqn:10} and \eqref{eqn:11}, we can conclude that
\begin{align}
&  \text{\small $\|q(\alpha u)\|\Bigg (\|q(\alpha u)\|^2\Big (1-\alpha b(\alpha) \Big ) - \half\! \E\|X-\langle X,u\rangle u\|^2  
- \half \!\!\sum \Big[\|q(\alpha u)\|^2\beta_k^2 (\alpha ) - |\E\langle X,w_k\rangle |^2\Big ]\Bigg )$} \nonumber\\
& \underset{\alpha\to 1}{\longrightarrow} \, f_X(u) - \sum_{k=1}^{d-1} \E\langle X,w_k\rangle\,\,\text{cov}(\langle X,u\rangle , \langle X,w_k\rangle. \label{eqn:12}
\end{align}
But, notice that the LHS of Equation~\eqref{eqn:12} can be rewritten as 
$$\|q(\alpha u)\|\Bigg (\|q(\alpha u)\|^2\half\Big (1-\alpha^2+( b(\alpha)-\alpha)^2 \Big ) - \half \sum_{k=1}^{d-1} \text{var} \langle X,w_k \rangle\Bigg),$$
by observing that $(1-b^2(\alpha))=\sum_{k=1}^{d-1} \beta^2_k(\alpha)$. 

Additionally, $\|q(\alpha u)\|^3(\alpha - b(\alpha))^2 \underset{\alpha\to 1}{\to} 0$, since $\|q(\alpha u)\|(\alpha -b(\alpha))$ converges to $0$ and $\|q(\alpha u)\|^2(\alpha -b(\alpha))$ converges to a finite limit, as $\alpha \to 1$, as a result of 
%Equations \eqref{eqn:rate-gird-1} and \eqref{eqn:rate-gird-2}, respectively, of Lemma \ref{lem:rate-gird}
Lemmas 6.2 and 6.4 in \cite{Girard2017}, respectively.
Therefore, we deduce that
\begin{eqnarray*}
&& \|q(\alpha u)\|\Bigg (\|q(\alpha u)\|^2 \Big (1-\alpha\Big ) - \left( \text{tr} \Sigma - \langle\Sigma u,u\rangle \right)\Bigg) \\
&& \underset{\alpha\to 1}{\longrightarrow} f_X(u) - \sum_{k=1}^{d-1} \E\Big(\langle X,w_k\rangle\Big)\,\,\text{cov}\Big(\langle X,u\rangle , \langle X,w_k\rangle\Big).
\end{eqnarray*}
Collating all the terms, we have
	\begin{eqnarray*}
        &&  \|q_X(\alpha u)\|\Big[\|q_X(\alpha u)\|^2 (1 - \alpha ) -  \half (\text{tr } \Sigma - \langle\Sigma u,u\rangle)\Big] \, \underset{\alpha\to 1}{\longrightarrow} \nonumber \\
        && \qquad \E\Big(\langle X,u\rangle \big[\|X-\langle X,u\rangle u\|^2 - \langle X,\E\left(X-\langle X,u\rangle u\right)\rangle \big]\Big) \\
        && \qquad -\sum_{k=1}^{d-1} \text{cov}(\langle X,u\rangle, \langle X,w_k\rangle)\, \E(\langle X,w_k\rangle ),
        	\end{eqnarray*}
which can further be simplified to
	\begin{eqnarray*}
        &&  \|q_X(\alpha u)\|\Big[\|q_X(\alpha u)\|^2 (1 - \alpha ) -  \half (\text{tr } \Sigma - \langle\Sigma u,u\rangle)\Big] \, \underset{\alpha\to 1}{\longrightarrow} \nonumber \\
        && \E\Big(\langle X,u\rangle \big[\|X-\langle X,u\rangle u\|^2\Big) - \langle (2\Sigma + \mu\mu^T)u, (I - uu^T)\mu \rangle . \hfill \qed
        	\end{eqnarray*}

%%%%%%%%%%%%%%%%
\section{Conclusion}
%%%%%%%%%%%%%%%%

In this paper, we investigated the extremal behaviour of geometric quantiles in multivariate settings. Our main contributions include the derivation of general lower and upper bounds for the growth rate of the geometric quantile norm without moment assumptions, as well as the identification of connections between multivariate geometric quantiles, univariate quantiles, and Tukey depth. We provided geometric insights and analyzed asymptotic properties, highlighting relationships with half-space depth and refinements under integrability (moment) conditions. More generally, our geometry-based approach extends previous results by providing bounds that do not rely on moments or smoothness assumptions, and therefore apply to a broad class of distributions, including those with heavy tails.
Beyond the general moment-free bounds, we also considered the integrable case and extended the asymptotic expansion approach of \cite{Girard2017} to higher orders, in order to capture skewness and higher-order tail effects. 
Overall, our study provides a deeper understanding of the geometry of multivariate quantiles.

% \section*{Acknowledgment}
% This work is based on part of the doctoral thesis of the first author, completed while he was a student at TIFR–CAM under the supervision of the two other authors. It benefited greatly from mutual visits at ESSEC Business School, Paris, France, and at TIFR–CAM, Bangalore, India. The authors are grateful to the hosting institutions for their support and to the {\it Fondation des Sciences de la Mod\'elisation} (ANR-11-LABX-0023-01), and SERB--MATRICS grant MTR/2020/000629 for funding.

\bibliographystyle{abbrv}
\bibliography{LitSibsankar.bib}

~\vspace{3ex}

\appendix
{\Large \bf Appendix}
%\section{Proof of Theorem~\ref{thm:higher-order} and Directional Lemma}

\section{Geometric Quantile versus Halfspace Depth}
\label{ssec:Tukey-depth}

In this section, we investigate the tightness of the set inclusion established in Theorem~\ref{thm:lowerbd-q}. More precisely, given a geometric quantile region at level $\alpha$, we ask what is the largest halfspace depth region that is contained within it.

From an analytical standpoint, it is generally difficult to determine which halfspace depth contour best fits inside a given geometric quantile contour. Consequently, we adopt a numerical approach to explore this question.

Let $\displaystyle \partial G(\alpha)$ denote the geometric quantile contour, i.e., the boundary of the geometric quantile region $\displaystyle G(\alpha)$ at level $\alpha$.
We define
$$
\alpha_{\mathrm{best}} = \max_{x \in \partial G(\alpha)} HD(x \mid X).
$$
By definition, the halfspace depth region $\displaystyle \mathcal D(\alpha_{\mathrm{best}})$ is the largest halfspace depth region contained in $\displaystyle G(\alpha)$.
In practice, $\alpha_{\mathrm{best}}$ is estimated by evaluating the halfspace depth on a dense grid of points along $\partial G(\alpha)$ and taking the maximum.

In Figure~\ref{fig:depth_contours}, we plot three contours: 
the geometric quantile contour $\partial G(\alpha)$, 
the numerically computed halfspace depth contour $\displaystyle \partial \mathcal D(\alpha_{\mathrm{best}})$, 
and the theoretical depth contour $\displaystyle \partial \mathcal D\!\left(\frac{1-\alpha^2}{M_\gamma}\right)$, 
for two distributions: a Gaussian distribution and a heavy-tailed distribution, i.e., a Student’s $t$ distribution coupled with a Gaussian copula.

Recall from \eqref{dfn:M} that
\[
M_\gamma = (1-\gamma)\inf_{q \in \S^{d-1}} \pr(U_0 \in A_q^\gamma).
\]
In the numerical experiments, we estimate $M_\gamma$ by fixing $\gamma = 0.01$ and minimizing $\pr(U_0 \in A_q^\gamma)$ over $q \in \S^{d-1}$. 
We then choose $\alpha$ such that
\[
\alpha^2 \in \left(1-\frac{M_\gamma}{d+1},\,1\right),
\]
which ensures that the geometric quantile region at level $\alpha$ contains the halfspace depth median, as required by Theorem~\ref{thm:lowerbd-q}. 
If this condition is violated, the geometric quantile region may fail to contain any nonempty halfspace depth region.

\begin{rem}
From a numerical perspective, it is straightforward to check whether one quantile region is contained within another. 
Given a geometric quantile region, one can evaluate the halfspace depth at points inside the region and take the minimum value attained. 
Any halfspace depth region with depth smaller than this minimum is then guaranteed to be contained in the geometric quantile region.

From an analytical standpoint, however, general containment relationships between depth regions and quantile regions remain largely open. 
Theorem~\ref{thm:lowerbd-q} provides a partial analytical link between geometric quantile regions and halfspace depth regions, where the halfspace depth arises naturally via its interpretation as a probability bound on directional events; see Equation~\eqref{eqn:estimate_index_vector}.
\end{rem}

The geometric inclusion underlying Theorem~\ref{thm:lowerbd-q}, and its consequence
stated in Theorem~\ref{thm:lowerbd-q-univ}, are illustrated in
Figure~\ref{fig:depth_contours}, where we consider two examples:
\begin{enumerate}
    \item {\it Bivariate Gaussian distribution} with standard Gaussian margins
    and correlation $\rho=0.7$; see Figure~\ref{fig:Gaussian}.
    \item {\it Heavy-tailed distribution lacking MRV structure:}
    Student's $t$ margins with different degrees of freedom,
    $\nu_1=1.2$ and $\nu_2=2.2$, coupled via a Gaussian copula with $\rho=0.7$;
    see Figure~\ref{fig:correlated}.
\end{enumerate}

\begin{figure}[htbp]
    \centering
    \begin{subfigure}[b]{0.47\textwidth}
        \centering
        \includegraphics[width=\textwidth]{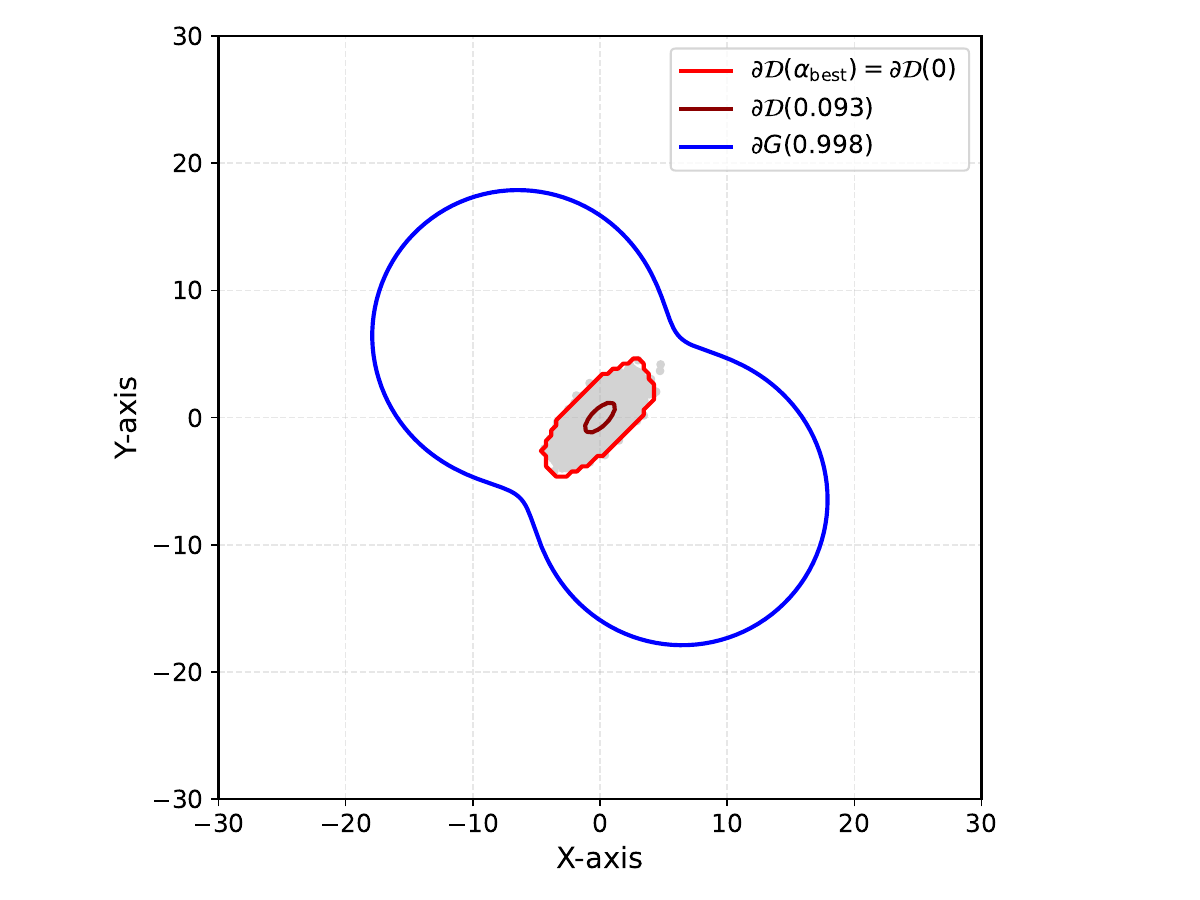}
        \caption{\small Bivariate Gaussian distribution with standard Gaussian
        margins and $\rho=70\%$}
        \label{fig:Gaussian}
    \end{subfigure}
    \hfill
    \begin{subfigure}[b]{0.49\textwidth}
        \centering
        \includegraphics[width=\textwidth]{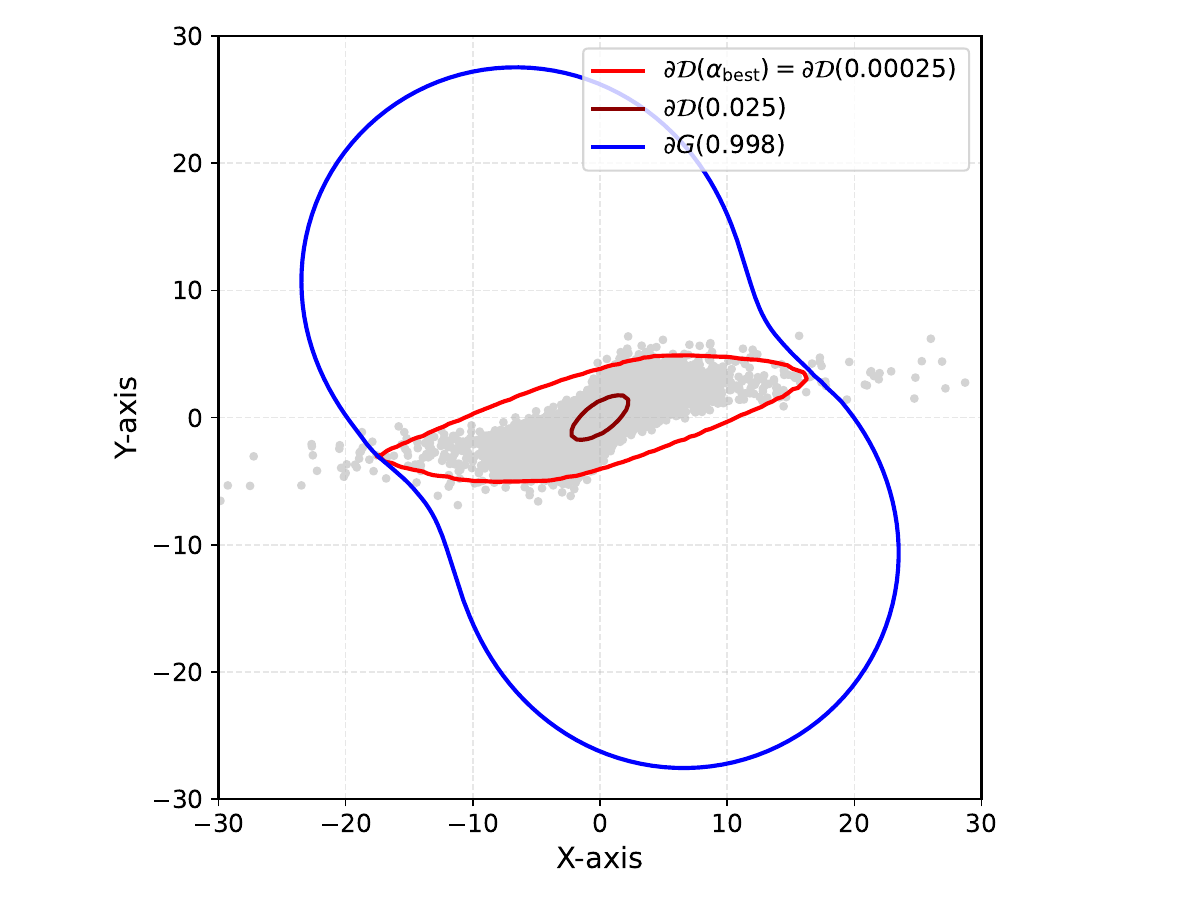}
        \caption{\small Student's $t$ marginals with $\nu_1=3$, $\nu_2=10$, coupled
        via a Gaussian copula ($\rho=70\%$)}
        \label{fig:correlated}
    \end{subfigure}
    \caption{Depth contours for different bivariate distributions}
    \label{fig:depth_contours}
\end{figure}

In addition to the comparison of the three contours, the figures also highlight
the different shapes of the geometric quantile regions.
In the case of the Gaussian copula (first row), the geometric quantile region
exhibits a pronounced elongated shape, whereas for the heavy-tailed example
(second row) the contour appears more rounded.

\newpage
%%%%%%%%%%%%%%%%%%%%%%%%%%%%%%%%%%%%%%%%%%%
\section{Rate of Convergence for the Direction of Geometric Quantiles}
\label{app:direction}
%%%%%%%%%%%%%%%%%%%%%%%%%%%%%%%%%%%%%%%%%%%

In this appendix, we present a complementary result on the direction of geometric quantiles, rather than their magnitude, the main focus of our paper, to provide completeness and extend \cite{Girard2017} by establishing a rate of convergence.
\\[1ex]
{\bf Lemma~A.}
{\it Let $u \in S^{d-1}$ and $\pr$ be a probability measure on $\real^d$ with no atom such that its support is not concentrated on any line. 
\\If $\E\|X\|^2<\infty$, and $\Sigma$ denotes the covariance matrix of $X$, then we have 
\begin{align*}
			&&\!\!\!\!\!\!\!\!\!\!\!\!\!\!\!\!\!\!\!\!\! \|q_X(\alpha u)\| \Big[ q_X(\alpha u)-\big\{\|q_X(\alpha u)\|u + \E(X - \langle X,u\rangle u)\big\}\Big]\qquad\\ 
 			&&\;  \underset{\alpha\to 1}{\longrightarrow}\, 
 -\half\| \E(X-\langle X,u\rangle u)\|^2 u + (I- u u^T)\Sigma u.
\end{align*}
Furthermore, if $X$ is assumed to be mean zero, then the limit simplifies to
$$
\|q_{X}(\alpha u)\| \left( q_{X}(\alpha u) - \|q_{X}(\alpha u)\|u\right) \,\underset{\alpha\to 1}{\longrightarrow}\, (I-uu^T)\Sigma u.
$$
}
\\[1ex]
\begin{proof}
%%%
Using the basis representation \eqref{eqn:decomposition}, we can write
\begin{eqnarray}\label{eqn:part-I-II}
&& \|q(\alpha u)\|\, \Big[q(\alpha u) - \|q(\alpha u)\| \, u - \E(X - \langle X,u\rangle u) \Big] \nonumber\\
&& + \half\| \E(X-\langle X,u\rangle u)\|^2 u - (I-uu^T)\Sigma u\nonumber\\
& = & \underbrace{\left[\|q(\alpha u)\|^2 \left(b(\alpha ) - 1\right) + \half\| \E(X-\langle X,u\rangle u)\|^2 \right]}_{I} u \\
&& \mbox{} + \underbrace{\sum _{k-1}^{d-1} \Big(\|q(\alpha u)\|^2 \,\beta_k(\alpha) - \|q(\alpha u)\|\,\E\langle X,w_k\rangle -  
\text{cov}(\langle X,u\rangle , \langle X, w_k\rangle)\Big)  \, w_k.}_{II} \nonumber
\end{eqnarray}

Using Lemma 6.4 in \cite{Girard2017}, we observe that
\begin{equation}\label{eqn:part-II}
\|q(\alpha u)\|^2 \beta_k(\alpha) - \|q(\alpha u)\| \E\langle X,w_k\rangle - \text{cov}(\langle X,u\rangle , \langle X,w_k\rangle) \,\underset{\alpha\to 1}{\longrightarrow}\, 0,
\end{equation}
concluding that part--$II$ in \eqref{eqn:part-I-II} converges to $0$ as $\alpha \to 1$. 

Next, let us consider part--$I$ in \eqref{eqn:part-I-II}. Recall that, $\|q(\alpha u)\|^2(1-\alpha) \underset{\alpha\to 1}{\longrightarrow} \frac12 (\text{tr}\Sigma - \langle \Sigma u, u\rangle)$, and $\|q(\alpha u)\|^2(\alpha b(\alpha)-1) \underset{\alpha\to 1}{\longrightarrow}  - \frac12 \E \|X-\langle X,u\rangle u\|^2$, by Theorem 2.2 and  \cite[Lemma 6.3]{Girard2017}, respectively. Hence, we can conclude that
\begin{equation}
\|q(\alpha u)\|^2 (1-b(\alpha)) \,\,\, \underset{\alpha\to 1}{\rightarrow} \,\,\, \frac12 \sum_{k=1}^{d-1} \left[\E\langle X-\langle X,u\rangle u,w_k\rangle\right]^2 = \frac12 \|\E\left(X-\langle X,u\rangle u\right) \|^2,
\end{equation}
which proves part-$I$ of \eqref{eqn:part-I-II}. 

Therefore, we obtain:
\begin{align*}
			&&\!\!\!\!\!\!\!\!\!\!\!\!\!\!\!\!\!\!\!\!\! \|q_X(\alpha u)\| \Big[ q_X(\alpha u)-\big\{\|q_X(\alpha u)\|u + \E(X - \langle X,u\rangle u)\big\}\Big]\qquad\\ 
 			&&\;  \underset{\alpha\to 1}{\longrightarrow}\, 
 -\half\| \E(X-\langle X,u\rangle u)\|^2 u + (I- u u^T)\Sigma u.
\end{align*}
By substituting $X-\E(X)$ for $X$, the lemma follows directly in a simplified form, making its interpretation straightforward.

\end{proof}

\end{document}